\documentclass[smallextended]{svjour3}
\smartqed  
\usepackage{amsfonts}
\usepackage{amsmath}
\usepackage{graphicx}
\usepackage{epsfig}
\usepackage{amssymb}
\usepackage{tikz}
\usetikzlibrary{positioning}
\usetikzlibrary{shapes}

\usepackage{comment}

\journalname{Graphs and Combinatorics}

\begin{document}
\title{On Dirac and Motzkin problem in discrete geometry}
\author{Jan~Florek}
\institute{J.~Florek \at Faculty of Pure and Applied Mathematics,
 Wroclaw University of Science and Technology,
 Wybrze\.{z}e Wyspia\'nskiego 27,
50--370 Wroc{\l}aw, Poland\\
\email{jan.florek@pwr.edu.pl}}

\date{Received: date / Accepted: date}

\maketitle
\begin{abstract} Dirac and Motzkin conjectured that any set $\mathcal{X}$ of $n$ non-collinear points in the plane has an element incident with at least $\lceil \frac{n}{2} \rceil$ lines spanned by~$\mathcal{X}$.  In this paper we prove that any set  $\mathcal{X}$ of $n$ non-collinear points in the plane, distributed on three lines passing through a common point, has an element incident with at least $\lceil \frac{n}{2} \rceil$ lines spanned by $\mathcal{X}$.


\end{abstract} 


\section{Introduction}
We consider a set $\mathcal{X}$ of  $n$ non-collinear points in the plane ($n \geqslant 3$), and  the set $\mathcal{L}$ of lines  \emph{spanned} by $\mathcal{X}$, i.e., the lines passing through at least two points of $\mathcal{X}$. Let $t(\mathcal{X})$ be the maximum number of lines in $\mathcal{L}$ incident with a point of $\mathcal{X}$. Dirac \cite{dirac4} and Motzkin \cite{motzkin9} conjectured that for any set $\mathcal{X}$ of $n$ non-collinear points in the plane, we have $t(\mathcal{X})  \geqslant \lceil\frac{n}{2} \rceil$. Some counterexamples were shown for small values of $n$ (9, 15, 19, 25, 31, 37) by Gr\"{u}nbaum \cite{grunbaum6} (pp. 24--25)  (see also  Gr\"{u}nbaum~\cite{grunbaum7}). An infinite family of counterexamples was given by Felsner \cite{brass3} (p. 313). He constructed a set $\mathcal{X}$ of $n = 6k + 7$ ($k \geqslant 2$) non-collinear points in the plane for which $t(\mathcal{X})  < \lceil\frac{n}{2} \rceil$. Akiyama, Ito, Kobayashi and Nakamura \cite{akiyama1} constructed a set $\mathcal{X}$ of  $n$ non-collinear points in the plane satisfying  $t(\mathcal{X})  < \lceil\frac{n}{2} \rceil$, for every $n \geqslant 8$ except possibly $n = 12k +11$ where $k \geqslant 4$.
It is interesting that some of these sets (for $n = 4k +1$ with $k \geqslant 2$, see Fig.~1)  are distributed on three lines without a common intersection point.
Moreover, another one of these sets (for $n = 4k +1$ with $k \geqslant 2$, see Fig.~2) are distributed on three lines passing through a common point, 
except for one point (see also Florek \cite{florek5}).
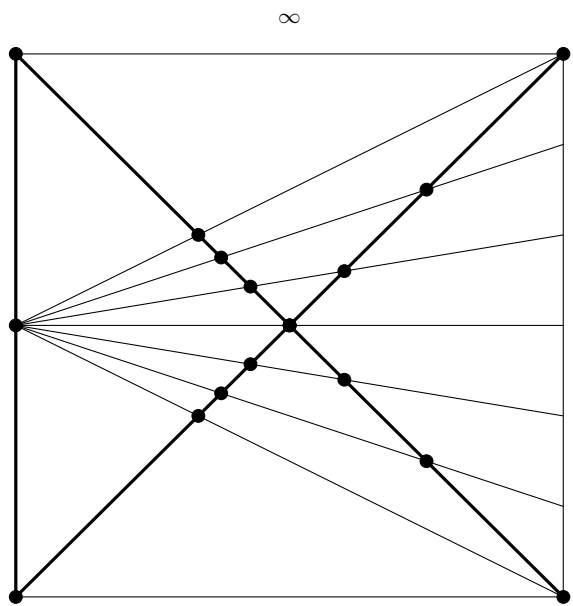
\begin{figure}[htb] 
\begin{center}
\begin{tikzpicture}[scale=1.2]
\draw[very thick] (6,-3)--(0,3)--(0,-3)--(6,3);
\draw (0,3)--(6,3)--(6,-3)--(0,-3);
\draw (0,0)--(6,3);
\draw (0,0)--(6,2);
\draw (0,0)--(6,1);
\draw (0,0)--(6,0);
\draw (0,0)--(6,-1);
\draw (0,0)--(6,-2);
\draw (0,0)--(6,-3);
\coordinate (00) at (0,0);
\coordinate (A0) at (0,3);
\coordinate (A1) at  (intersection of 0,3--6,-3
       and 0,0--6,3);
\coordinate (A2) at  (intersection of 0,3--6,-3
       and 0,0--6,2);
\coordinate (A3) at  (intersection of 0,3--6,-3
       and 0,0--6,1);
\coordinate (A4) at  (intersection of 0,3--6,-3
       and 0,0--6,0);
\coordinate (A5) at  (intersection of 0,3--6,-3
       and 0,0--6,-1);
\coordinate (A6) at  (intersection of 0,3--6,-3
       and 0,0--6,-2);
\coordinate (A7) at  (intersection of 0,3--6,-3
       and 0,0--6,-3);
         
\coordinate (B0) at (0,-3);
\coordinate (B1) at  (intersection of 0,-3--6,3
       and 0,0--6,3);
\coordinate (B2) at  (intersection of 0,-3--6,3
       and 0,0--6,2);
\coordinate (B3) at  (intersection of 0,-3--6,3
       and 0,0--6,1);
\coordinate (B4) at  (intersection of 0,-3--6,3
       and 0,0--6,0);
\coordinate (B5) at  (intersection of 0,-3--6,3
       and 0,0--6,-1);
\coordinate (B6) at  (intersection of 0,-3--6,3
       and 0,0--6,-2);
\coordinate (B7) at  (intersection of 0,-3--6,3
       and 0,0--6,-3);
       
\draw (3,3) node[anchor=south, yshift=3mm]{$\infty$};       

\filldraw [fill=black, draw=black]
(00) circle (2pt)   
       
(A0) circle (2pt)
(A1) circle (2pt)
(A2) circle (2pt)
(A3) circle (2pt)
(A4) circle (2pt)
(A5) circle (2pt)
(A6) circle (2pt)
(A7) circle (2pt)

(B0) circle (2pt)
(B1) circle (2pt)
(B2) circle (2pt)
(B3) circle (2pt)
(B4) circle (2pt)
(B5) circle (2pt)
(B6) circle (2pt)
(B7) circle (2pt);

\end{tikzpicture}
\caption{Akiyama, Ito, Kobayashi and Nakamura construction for $n=17$ points (including~$\infty$)}
\end{center}
\end{figure}

\begin{figure}[!hb]
\begin{center}

\begin{tikzpicture}[scale=1.7]

\coordinate (A0) at (0,0);
\coordinate (A2) at (2,0);
\coordinate (A4) at (4,0);
\coordinate (A6) at (6,0);

\coordinate (B0) at (0,1);
\coordinate (B1) at (1,1);
\coordinate (B2) at (2,1);
\coordinate (B3) at (3,1);
\coordinate (B4) at (4,1);
\coordinate (B5) at (5,1);
\coordinate (B6) at (6,1);

\coordinate (C0) at (0,2);
\coordinate (C2) at (2,2);
\coordinate (C4) at (4,2);
\coordinate (C6) at (6,2);

\draw[very thick] (A0)--(A6);
\draw[very thick] (B0)--(B6);
\draw[very thick] (C0)--(C6);

\draw (A0)--(C0)--(A2)--(C2)--(A4)--(C4)--(A6)--(C6)--(A4)--(C4)--(A2)--(C2)--cycle;

\filldraw[black]
(A0) circle (2pt) 
(A2) circle (2pt) 
(A4) circle (2pt) 
(A6) circle (2pt) 

(B0) circle (2pt)
(B1) circle (2pt)
(B2) circle (2pt)
(B3) circle (2pt)
(B4) circle (2pt)
(B5) circle (2pt)
(B6) circle (2pt)

(C0) circle (2pt)
(C2) circle (2pt)
(C4) circle (2pt)
(C6) circle (2pt);

\draw (B6) node[anchor=west, yshift=0mm, xshift=3mm]{$\infty_1$};
\draw (3,2) node[anchor=south, yshift=3mm]{$\infty_2$};
\end{tikzpicture}
\caption{Akiyama, Ito, Kobayashi and Nakamura construction for $n=17$ points (including $\infty_1$ and  $\infty_2$) distributed, except for $\infty_2$, on three lines passing through $\infty_1$.}
\end{center}
\end{figure}
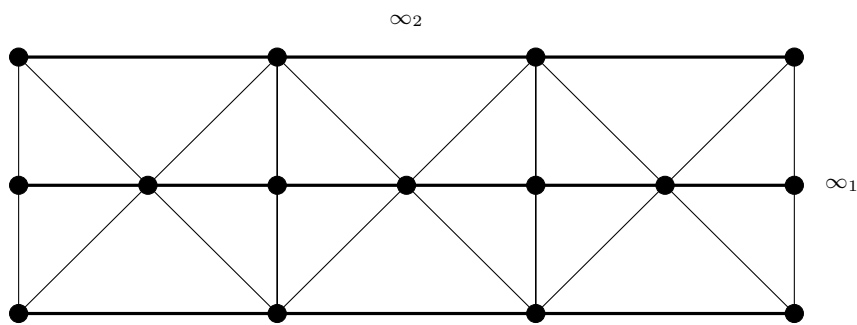

In this paper we prove Theorem 1: any set $\mathcal{X}$ of  $n$ non-collinear points in the plane, distributed on three lines passing through a common point, has an element incident with at least $\lceil \frac{n}{2} \rceil$ lines spanned by $\mathcal{X}$. 

The construction of the set $\mathcal{X}$ in the projective plane given in Fig.~1 and Fig.~2 is clear due to the following lemma (see  Gr\"{u}nbaum \cite{grunbaum6} and  Akiyama, Ito, Kobayashi and Nakamura \cite{akiyama1} (Lemma $D$)): Let $\mathcal{X}$ be a finite set of points in the projective plane such that $\mathcal{X}$ has points at infinity, and let $\mathcal{L}$ be the set of lines spanned by $\mathcal{X}$.  Then the arrangements $(\mathcal{X},\mathcal{L})$ can be transformed to an arrangement in the Euclidean plane.

The ``weak Dirac conjecture'' proved by Beck \cite{beck2} and independently by Szemer\'{e}di and Trotter \cite{szemeredi10} states that there is a constant $c> 0$ such that in every non-collinear set $\mathcal{X}$ of $n$ points in the plane some element is incident with at least $cn$ lines spanned by $\mathcal{X}$. Brass, Moser and Pach \cite{brass3} (p.~313) proposed the following ``strong Dirac conjecture'': there is a constant $c > 0$ such that any set~$\mathcal{X}$ of $n$ points in the plane, not all on a line, has an element which lies on at least $\lceil \frac{n}{2} \rceil - c$ lines spanned by $\mathcal{X}$. Klee and Wagon  \cite{klee8} (p. 29) conjectured that for any set $\mathcal{X}$ of $n$ non-collinear points in the plane, we have $t(\mathcal{X})  \geqslant \lceil \frac{n}{3} \rceil$.
\section{Main result}\label{sec1}

Let $\mathcal{X}$ be a set of $n$ non-collinear points in the plane. A line passing through exactly two points of $\mathcal{X}$ is called an \emph{ordinary line}. A point of $\mathcal{X}$ incident with at least  $\lceil \frac{n}{2} \rceil$ lines spanned by  $\mathcal{X}$ is called an \emph{ordinary point} (Motzkin \cite[p. 452]{motzkin9}). Let $(A,B)$ denote an open segment of a line with end-vertices $A$ and $B$.
\begin{theorem}\label{theorem1}
Any set $\mathcal{X}$ of  $n$ non-collinear points in the plane distributed on three lines passing through a point $A$ $($$A$ is not required to be in the set $\mathcal{X}$$)$ has an ordinary point.
\end{theorem}

\begin{proof}
Let $p$, $q$, $r$ be three lines passing through a common point $A$ and suppose $p_1$,  $p_2 \subset p$ ($q_1$,  $q_2 \subset q$ and $r_1$,  $r_2 \subset r$) are open half lines originating at the point~$A$. We can assume that $p_1$, $q_1$, $r_1$, $p_2$, $q_2$, $r_2$ are successive half lines around the point $A$. On every half line $p_1$, $q_1$, $r_1$, $p_2$, $q_2$, $r_2$ passing through a point of~$\mathcal{X}$, we choose respectively a point $P_1$, $Q_1$, $R_1$, $P_2$, $Q_2$, $R_2$ of $\mathcal{X}$, different from $A$, which is closest to $A$. The points $P_1$ and $P_2$ ($Q_1$ and $Q_2$, $R_1$ and $R_2$, respectively) does not exist simultaneously.

Suppose that there are three non-successive in pairs half lines (say  $p_{1}, r_{1}, q_{2}$) such that $\mathcal{X} \subseteq {p_1} \cup {r_1} \cup {q_2} \cup \{A\}$. Notice that $p_1$ and $r_1$ are lying on different sides of the line $q$. Similarly, $r_1$ and $q_2$ ($q_2$ and $p_1$) are lying on different sides of the line $p$ ($r$, respectively). Hence every line spanned by $\mathcal{X}$, except for $p$, $q$,~$r$, is ordinary. Notice that at least two half lines (say $p_1$, $r_1$) pass through an element of $\mathcal{X}$, because $\mathcal{X}$ is not on a line. Hence,  points $P_1$ and $R_1$ exist. Notice that if $p_{1}$ has at least $\lceil \frac{n}{2} \rceil$ points of $\mathcal{X}$, then the point $R_1$ is ordinary. Otherwise ${r_1} \cup {q_2} \cup \{A\}$ has at least $\lceil \frac{n}{2} \rceil$ points of $\mathcal{X}$ and the point $P_1$ is ordinary.

Suppose now that there are two successive half lines (say $p_{1}$ and $q_{1}$) each of which cross the set $\mathcal{X}$. Hence,  the points $P_1$ and $Q_1$ exist. Certainly, if $p$ and~$q$ are both not spanned by $\mathcal{X}$, then points $P_1$ and $Q_1$ are ordinary.  Let $p$ or $q$  be spanned by $\mathcal{X}$. Suppose that $q$ is spanned by $\mathcal{X}$ and $P_1$ is the only element of $\mathcal{X}$ incident with $p$. Then, $A \notin \mathcal{X}$. Hence, if sets $\mathcal{X}\cap q$, $\mathcal{X}\cap r$ are of the same (different) cardinality, then the point $Q_1$ ($P_1$, respectively) is ordinary. By analogy, if $p$ is spanned by $\mathcal{X}$ and $Q_1$ is the only element of $\mathcal{X}$ incident with $q$, then $P_1$ or $Q_1$ is ordinary. Hence, we may assume that
\\[4pt]
($i$) \quad $p$ and $q$ are both spanned by $\mathcal{X}$. 
\\[4pt]
Let us consider the following cases:
 \begin{enumerate}
\item[($a$)]  the line $P_1Q_1$ crosses the half line $r_1$,
\item[($b$)] the line $P_1Q_1$ crosses the half line $r_2$,
\item[($c$)] the lines $P_1Q_1$ and $r$ are parallel.
\end{enumerate}

Case ($a$). We prove, by induction on $|\mathcal{X}|$,  that $P_1$ or $Q_1$ is ordinary. Let $B$ be the intersection point of the line $P_1Q_1$ with the half line $r_1$. 

Suppose that $\mathcal{X}\cap (A,B)$ is empty. Since $\mathcal{X}\cap (A,P_1)$ is empty, we have:
\\[4pt]
($ii$) \quad for every $S\in \mathcal{X}\cap({p\cup r}) \setminus \{A, P_{1}, B\}$ the line $Q_{1}S$ is ordinary. 
\\[4pt]
If $p\cup r$ has at least  $\lceil \frac{n-4}{2} \rceil$ points of $\mathcal{X} \setminus \{A, B, P_{1}, Q_{1}\}$, then, by ($i$) and $(ii)$ the point $Q_1$ is ordinary.  If the line $q$ has at least $\lceil \frac{n-4}{2} \rceil$ points of $\mathcal{X} \setminus \{A, B, P_{1} ,Q_{1}\}$, then, by $(i)$, the point $P_1$ is ordinary.

Let now the set $\mathcal{X}\cap(A,B)$ has $k > 0$ points. Let $(A,B)_0$ be the set of all points $T \in (A,B)$ for which the lines $Q_{1}T$ and $p$ are not parallel, and suppose that $\phi(T)$ is the intersection point of the line $Q_{1}T$ with the line $p$, for every $T \in (A,B)_0$. If we delete from the set $\mathcal{X}$ all elements of $(A,B)$ and all elements of $\phi(\mathcal{X}\cap(A,B)_0)$ we obtain a set $\mathcal{X}'$ with at least $n-2k$ elements. Since $\mathcal{X}\cap q \subset \mathcal{X}'$ and $P_{1}\in \mathcal{X}'$, ($i$) shows that $\mathcal{X}'$ is not all on a line. Hence, by induction, $P_1$ or $Q_1$ is incident with at least $\lceil \frac{n}{2}\rceil - k$ lines spanned by $\mathcal{X}'$. Since $\mathcal{X}\cap(A,Q_1)$, $\mathcal{X}'\cap(A,B)$ and $\mathcal{X}'\cap\phi((A,B)_0)$ are empty sets, the point $P_1$ ($Q_1$) is the only point of the set $\mathcal{X}'$ incident with the line $P_{1}T$ ($Q_{1}T$, respectively) for $T\in{\mathcal{X}\cap(A,B)}$. Hence, the above lines are not spanned by~$\mathcal{X}'$. Since $\mathcal{X}\cap(A,B)$ has $k$ points, $P_1$ or $Q_1$ is incident with at least $\lceil \frac{n}{2} \rceil$ lines spanned by $\mathcal{X}$.

Case ($b$). The proof of case ($b$) is analogous to that of ($a$).

Case ($c$). Since $\mathcal{X}\cap (A,Q_1)$ and $\mathcal{X}\cap (A,P_1)$ are empty sets, we have:
\\[4pt]
$(iii)$ \quad the line $P_{1}S$ is ordinary,  for every $S \in \mathcal{X}\cap r_{1}$,
\\[4pt]
$(iv)$ \quad the line $Q_{1}S$ is ordinary, for every $S\in \mathcal{X}\cap r_{2}$.
\\[4pt]
If  $q \cup r_1$ has at least $\lceil \frac{n-1}{2} \rceil$ points of $\mathcal{X} \setminus \{A\}$, then, by ($i$) and ($iii$), the point $P_1$ is ordinary. If  $p \cup r_{2}$ has at least $\lceil \frac{n-1}{2} \rceil$ points of $\mathcal{X} \setminus \{A\}$, then, by ($i$) and ($iv$), the point $Q_1$ is ordinary.
\qed
\end{proof}

The author declare that no funds, grants, or other support were received during the preparation of this manuscript.
\end{document}